\documentclass[11pt,a4paper]{amsart}
\usepackage{amsmath}
\usepackage{amssymb}
\usepackage{amsthm}
\usepackage{latexsym}
\usepackage{amsbsy}
\usepackage[all]{xypic}
\usepackage{amscd}
\usepackage{mathrsfs}
\usepackage{tipa}
\usepackage{txfonts}
\usepackage{amsfonts}
\usepackage{graphicx}
\usepackage{listings}
\usepackage{url}
\usepackage{bookmark}
\usepackage{tikz}

\newtheorem{theorem}{Theorem}[section]
\newtheorem{lemma}[theorem]{Lemma}
\newtheorem{corollary}[theorem]{Corollary}
\newtheorem{proposition}[theorem]{Proposition}
\newtheorem{definition}[theorem]{Definition}
\newtheorem{remark}[theorem]{Remark}

\title{A Silting Theorem}

\begin{document}

\providecommand{\ann}{\mathop{\rm ann}\nolimits}%
\providecommand{\gld}{\mathop{\rm gl. dim}\nolimits}%
\providecommand{\gord}{\mathop{\rm gor. dim}\nolimits}%
\providecommand{\pd}{\mathop{\rm pd}\nolimits}%
\providecommand{\rk}{\mathop{\rm rk}\nolimits}%
\providecommand{\Fac}{\mathop{\rm Fac}\nolimits}%
\providecommand{\ind}{\mathop{\rm ind}\nolimits}%
\providecommand{\Sub}{\mathop{\rm Sub}\nolimits}%
\providecommand{\coker}{\mathop{\rm coker}\nolimits}%
\providecommand{\cone}{\mathop{\rm cone}\nolimits}%

\def\s{\stackrel}
\def\A{\mathcal{A}}
\def\C{\mathcal{C}}
\def\DA{{D^b(A)}}
\def\H{\mathcal{H}}
\def\T{\mathcal{T}}
\def\K{\mathcal{K}}
\def\P{\mathcal{P}}
\def\X{\mathcal{X}}
\def\Y{\mathcal{Y}}
\def\F{\mathcal{F}}
\def\R{\mathcal{R}}
\def\L{\mathcal{L}}
\def\CP{\C(P^\bullet)}
\def\op{\text{op}}
\providecommand{\add}{\mathop{\rm add}\nolimits}%
\providecommand{\ann}{\mathop{\rm ann}\nolimits}%
\providecommand{\End}{\mathop{\rm End}\nolimits}%
\providecommand{\Ext}{\mathop{\rm Ext}\nolimits}%
\providecommand{\Hom}{\mathop{\rm Hom}\nolimits}%
\providecommand{\inj}{\mathop{\rm inj}\nolimits}%
\providecommand{\per}{\mathop{\rm per}\nolimits}%
\providecommand{\proj}{\mathop{\rm proj}\nolimits}%
\providecommand{\rad}{\mathop{\rm rad}\nolimits}%
\providecommand{\soc}{\mathop{\rm soc}\nolimits}%
\providecommand{\thick}{\mathop{\rm thick}\nolimits}%
\providecommand{\Tr}{\mathop{\rm Tr}\nolimits}%
\renewcommand{\dim}{\mathop{\rm dim}\nolimits}%
\renewcommand{\Im}{\mathop{\rm Im}\nolimits}%
\renewcommand{\mod}{\mathop{\rm mod}\nolimits}%
\renewcommand{\ker}{\mathop{\rm ker}\nolimits}%
\renewcommand{\rad}{\mathop{\rm rad}\nolimits}%
\def \text{\mbox}
\def\ends{\end{enumerate}}
\newcommand{\id}{\operatorname{id}}
\renewcommand{\op}{\operatorname{op}}
\renewcommand{\k}{\mathbf{k}}
\newcommand{\p}{\mathbf{P}}
\newcommand{\q}{\mathbf{Q}}
\newcommand{\x}{\mathbf{X}}
\newcommand{\y}{\mathbf{Y}}
\renewcommand{\c}{\mathbf{C}}
\renewcommand{\d}{\mathbf{D}}
\newcommand{\e}{\mathbf{E}}
\renewcommand{\t}{\mathbf{T}}

\author[Buan]{Aslak Bakke Buan}
\address{
Department of Mathematical Sciences
Norwegian University of Science and Technology
7491 Trondheim
NORWAY
}
\email{aslakb@math.ntnu.no}

\author[Zhou]{Yu Zhou}
\address{
Department of Mathematical Sciences
Norwegian University of Science and Technology
7491 Trondheim
NORWAY
}
\email{yu.zhou@math.ntnu.no}

\begin{abstract}
We give a generalization of the classical tilting theorem of Brenner and Butler.
We show that for a 2-term silting complex $\p$ in the bounded homotopy category
$K^b(\proj A)$ of finitely generated projective modules of a finite dimensional algebra $A$, the algebra $B = \End_{K^b(\proj A)}(\p)$ admits a 2-term silting complex $\q$ with the
following properties: (i) The endomorphism algebra of $\q$ in
$K^b(\proj B)$ is a factor algebra of $A$, and
(ii) there are induced torsion pairs in $\mod A$ and $\mod B$,
such that we obtain natural equivalences induced by
$\Hom$- and $\Ext$-functors. Moreover, we show how the Auslander-Reiten
theory of $\mod B$ can be described in terms of the
Auslander-Reiten
theory of $\mod A$.
\end{abstract}

%\date{3 March 2015}

\thanks{
This work was supported by FRINAT grant number 231000, from the
Norwegian Research Council. Support by the Institut Mittag-Leffler (Djursholm, Sweden) is gratefully
acknowledged.
}

\maketitle

\section*{Introduction}

The fundamental idea of tilting theory is to relate the module
categories of two algebras by the use of tilting functors. Such
functors were introduced by Brenner and Butler, in \cite{bb}, who were
generalizing the ideas in \cite{bgp} and \cite{apr}.

In the seminal paper \cite{hr}, Happel and Ringel introduced the
concepts of {\em tilting modules} and {\em tilted algebras}.
A tilted algebra is the endomorphism ring of a tilting module over a
hereditary finite dimensional algebra.
Happel \cite{h} and Cline, Parshall, Scott \cite{cps} proved that
tilting modules induce derived equivalences, and inspired by this
Rickard \cite{rick} introduced the concept of {\em tilting complexes}, as a necessary ingredient in
developing Morita theory for derived categories.

Over the last 35 years these ideas and concepts have become an essential tool in many branches of mathematics,
including algebraic geometry, finite group theory, algebraic group theory and algebraic topology, see \cite{ahk}. More recently, the development
of cluster tilting theory, see \cite{k, rei}, has spurred further interest in the topic and the relation to cluster algebras \cite{fz}.

Let us briefly recall the main ideas from \cite{bb} and \cite{hr}.
Let $\k$ be a field, let $A$ be a finite dimensional algebra over $\k$, and $T$ a tilting module in $\mod A$, the category of
(finite dimensional) right $A$-modules. That is: $T$ is a module with projective dimension at most 1 ($\pd T \leq 1$),
with $\Ext_A^1(T,T)=0$ and such that $\left| T \right| = \left| A \right|$, where $\left| X \right|$ denotes the number of indecomposable
direct summand in $X$, up to isomorphism.
Let $B = \End_A(T)$. Then $D(T)_B$ is a cotilting module over $B$ and $A \cong \End_B(D(T)_B)$, where $D$ is the $\k$-dual of finite dimensional $\k$-vector spaces.
Cotilting modules are defined by replacing $\pd T \leq 1$ with $\id T \leq 1$ in the definition of tilting modules,
where $\id T$ is the injective dimension of $T$.
Moreover, let $\T = \Fac T$ be the full subcategory of $\mod A$ whose objects are generated by $T$, and
let $\F$ be the full subcategory of $\mod A$ with objects $X$ such that $\Hom_A(\T, X) = 0$. Then $(\T, \F)$ is
a torsion pair in $\mod A$. There is also a torsion pair $(\X,\Y)$ in $\mod B$, induced by the cotilting module $D(T)_B$, and
$\Hom$- and $\Ext$-functors induce inverse equivalences
of $\T$ with $\Y$ and of $\F$ with $\X$.
%If $A$ is hereditary, then the tilted algebra $B$ is of global dimension at most two, and each indecomposable modules
%has either projective dimension at most one, or injective dimension at most one.

We generalize these results to the following setting. We consider a 2-term silting complex $\p$ in the bounded homotopy category $K^b(\proj A)$ of finitely generated projective $A$-modules. This is just a map between projective $A$-modules, considered as a complex, with
the property that $\Hom_{K^b(\proj A)}(\p,\p[1]) = 0$, and such that $\p$
generates $K^b(\proj A)$.
Let $B= \End_{K^b(\proj A)}(\p)$.  It then turns out that $\mod A$ and $\mod B$
can be compared in a way very similar to the setting with tilting modules.

It is known that if $\p$ is a 2-term silting complex in $K^b(\proj A)$ then $H^0(\p)$ is a tilting $(A/\ann H^0(\p))$-module and $H^{-1}(\nu\p)$ is a cotilting $(A/\ann H^{-1}(\nu \p))$-module, where $\nu$ is the Nakayama functor. In particular, the Brenner-Butler tilting theorem and its dual apply in this setting. However, for a general 2-term silting complex $\p$, both $\End_A(H^0(\p))$ and $\End_A(H^{-1}(\nu \p))$ are factor algebras of $\End_{K^b(\proj A)}(\p)$, so they are much smaller than $\End_{K^b(\proj A)}(\p)$. Hence, in the general case, the Brenner-Butler tilting theorem does not give the expected result.

The concept of silting complexes originated from Keller and Vossieck \cite{kv}. In \cite{hkm}, the relation between 2-term silting complexes and  torsion pairs in module categories was first considered. They were mainly dealing with abelian categories with arbitrary coproducts,
but we adapt many of their results to our setting.

More recently, there have been several papers, starting with \cite{ai}, often focusing on various (combinatorial) properties on the set of silting complexes. Silting complexes correspond to
bounded $t$-structures having a heart which is
a length category, i.e. there are finitely many simples, and all
objects have finite length \cite{ky}.

The set of
2-term silting complexes has a natural structure of an ordered
exchange graph, and as beautifully summarized in \cite{by},
this gives links (expressed as isomorphisms of exchange
graphs, see the figure in their introduction) to a plenitude of other structures which have recently
been studied. Among these are support
$\tau$-tilting modules \cite{air} in the module category,
and certain bounded $t$-structures in the bounded derived category,
see \cite[Corollary~4.3]{by}.
Starting with a quiver $Q$, with no loops or oriented 2-cycles,
there is a corresponding cluster algebra $A_Q$,
\cite{fz}, and then we obtain
also a correspondence with the clusters in $A_Q$,
see \cite{air}.
Given $Q$ as above, and a potential, there is a
correspondence with certain bounded $t$-structures
in the finite-dimensional derived category of the corresponding
Ginzburg dg algebra \cite{by,kq}.

In this paper and the forthcoming paper
\cite{bz}, we consider the endomorphism algebras of 2-term silting complexes, which so far
have been less studied. These algebras are isomorphic to the 0-th cohomology of the corresponding differential graded endomorphism algebras.

The paper is organized as follows.
In the first section, we review some background and notation, and state the main results.
In Section 2, we consider links between silting theory, $t$-structures and
torsion pairs. In Section 3, we prove further properties of 2-term silting complexes, and
the main result is proved in Section 4. In Section 5 we apply the main result
to obtain some information about the AR-theory of the endomorphism ring of a
2-term silting complex, inspired by similar
results in classical tilting theory, see \cite{ass}.

\bigskip

\noindent \textbf{Acknowledgements:}
We would like to thank the anonymous referee for careful reading and for many remarks which improved the presentation of the paper.

\section{Background and main result}

Let $A$ be a finite dimensional $\k$-algebra, and  $\mod A$ the
category of finitely generated right $A$-modules. Let $\DA$ be the
bounded derived category, with shift functor $[1]$. Whenever we consider subcategories of $\mod A$ or $\DA$, they are assumed to be full and closed under isomorphism.
For an object $M$ in an additive category, let $\add M$ denote the additive closure, i.e. the full subcategory
generated by all direct summands of direct sums of copies of $M$.

Recall that a {\em torsion pair} in $\mod A$, is a pair $(\X,\Y)$ of subcategories of $\mod A$, with the properties
that
\begin{itemize}
\item[-] $\Hom_A(\X, Y)= 0$ if and only if $Y$ is in $\Y$, and
\item[-] $\Hom_A(X, \Y)= 0$ if and only if $X$ is in $\X$.
\end{itemize}

If $M$ is an object in $\mod A$, then there is an exact sequence,
$$0 \to tM \to M \to M/tM \to 0$$
called the {\em canonical sequence} of $M$, and
with $tM$ in $\X$ and with $M/tM$ in $\Y$. Let $\proj A$ denote
the full subcategory of $\mod A$ generated by the projective modules.
We consider 2-term complexes $\p$ in $K^b(\proj A)$. These are
complexes $\p= \{P^i\}$ with $P^i=0$ for $i\neq -1,0$. Such a complex is
called {\em presilting} if $\Hom_{K^b(\proj A)}(\p,\p[1])=0$ and {\em silting} if in addition $\thick \p=K^b(\proj A)$.
Here, for an object $\x$ in $K^b(\proj A)$,
we denote by $\thick \x$ the smallest triangulated
subcategory closed under direct summands containing
$\x$. A 2-term silting complex $\p$ is {\em tilting}, if
in addition $\Hom_{K^b(\proj A)}(\p, \p[-1]) =0$.

Let $\p$ be a 2-term silting complex, and
consider the full subcategories of $\mod A$ given by
\begin{itemize}
\item[\ ] $\T(\p) =\{X \in \mod A \mid \Hom_{\DA}(\p,X[1]) = 0\},$ and
 \item[\ ] $\F(\p) =\{Y \in \mod A \mid \Hom_{\DA}(\p,Y) = 0 \}.$
\end{itemize}

Note that if $\p$ is a projective presentation of a tilting module $T$, then $\p$ is quasi-isomorphic to its 0-th cohomology $T=H^0(\p)$. Hence $\T(\p)=\ker\Ext^1_A(T,-)$ and $\F(\p)=\ker\Hom_A(T,-)$ and these are the classes considered in classical tilting theory. 

Our main theorem is a generalization of the Brenner-Butler tilting theorem to 2-term silting complexes. Note that (a) is from \cite{hkm}, (b) is from \cite{w}, while (c) and (d) can be easily deduced from \cite{by}.

\begin{theorem}\label{Main1}
Let $\p$ be a 2-term silting complex in $K^b(\proj A)$, and let $B = \End_{\DA}(\p)$.
\begin{itemize}
\item[(a)] The pair $(\T(\p), \F(\p))$ is a torsion pair in $\mod A$.
\item[(b)] There is a triangle
 $$A \rightarrow \p' \s{f}\rightarrow \p'' \rightarrow A[1]$$
with $\p', \p'' $ in $\add \p$.

Consider the 2-term complex $\q$ in $K^b(\proj B)$ induced by the map
$$\Hom_{\DA}(\p, f) \colon \Hom_{\DA}(\p, \p') \rightarrow \Hom_{\DA}(\p, \p'').$$

\item[(c)] $\q$ is a 2-term silting complex in $K^b(\proj B)$.
\item[(d)]  There is an algebra epimorphism $\Phi_{\p} \colon A \to  \overline{A}= \End_{D^b(B)}(\q)$.
\item[(e)] $\Phi_\p$ is an isomorphism if and only if $\p$ is tilting.

Let $\Phi_{\ast}\colon \mod \overline{A} \hookrightarrow \mod A$ be the induced inclusion functor.
\item[(f)] The restriction of the functors $\Hom_{\DA}(\p, -)$ and $\Phi_{\ast} \Hom_{D^b(B)}(\q, - [1])$ to $\T(\p)$ and $\F(\q)$
is a pair of inverse equivalences.
\item[(g)] The restriction of the functors $\Hom_{\DA}(\p,- [1])$ and $\Phi_{\ast} \Hom_{D^b(B)}(\q, -)$ to $\F(\p)$ and $\T(\q)$
is a pair of inverse equivalences.
\end{itemize}
\end{theorem}

We give a brief explanation on how (c) and (d) can be deduced from \cite{by}. Let $\widetilde{B}$ be the differential graded endomorphsim algebra of $\p$. Then $B\cong H^0(\widetilde{B})$. So the canonical epimorphism $\widetilde{B}\to B$ induces an exact functor $\phi: \per \widetilde B\to \per B$. Note that $\per\widetilde{B}$ is equivalent to $K^b(\proj A)$ and $\per B=K^b(\proj B)$. Then by \cite[Proposition~A.3]{by}, the functor $\phi$ induces is a bijection between a certain set of silting complexes in $K^b(\proj A)$ and the set of 2-term silting complexes in $K^b(\proj B)$. We remark that the complex $\q$ in (c), which we will construct in Section~\ref{sec:theory}, corresponds (up to isomorphism) to $A$ (as a stalk complex in $K^b(\proj A)$) under this bijection. The epimorphism in (d) is then obtained by \cite[Proposition~A.5]{by}.

We will use the following notation. For any subcategory $\T$ of $\mod A$, an $A$-module $X$ in $\T$ is called
{\em $\Ext$-projective} in $\T$ if $\Ext_A^1(X,Y)=0$ for all $Y$ in $\T$;
dually, $X$ in $\T$ is called {\em $\Ext$-injective} in $\T$ if
$\Ext_A^1(Y,X)=0$ for all $Y$ in $\T$. Furthermore, we let
$\nu$ denote the Nakayama functor $\nu = D\Hom_A(-,A)$, which is an equivalence from $\proj A$ to the full subcategory $\inj A$ of $\mod A$ generated by the injective modules. Then $\nu$ induces an equivalence
\[\nu: K^b(\proj A)\rightarrow K^b(\inj A).\]
It is well known that there is an isomorphism
\[\Hom_\DA(\x,\nu \y)\cong D\Hom_\DA(\y,\x)\]
for any $\x,\y\in K^b(\proj A)$ (see e.g. \cite[Chapter 1, Section 4.6]{h}). Note that the derived Nakayama functor $\nu$ in general is not a Serre functor on the bounded derived category $\DA$, as the algebra $A$ is not assumed to have finite global dimension.

\section{2-term silting complexes, $t$-structures and torsion pairs}\label{sec:silting}

%The proof of Theorem \ref{Main1} will be carried out in this
%section, it will be completed in section \ref{sec:tst}. Inspired by similar
%results in classical tilting, see \cite{ass}, we will also give some
%applications towards comparing the AR-theory of an algebra obtained as
%the endomorphism ring of a silting object with the AR-theory of the
%original algebra. This will be carried out in section \ref{sec:ar}.

In this section we recall the notion of a $t$-structure \cite{bbd} in a
triangulated category,  and the interplay between $t$-structures, torsion pairs
and 2-term silting complexes.

A pair $(\X,\Y)$ of subcategories of $\DA$ is called a {\em $t$-structure} if and only if the following conditions hold:
\begin{description}
  \item[(1)] $\X[1]\subset\X$ and $\Y[-1]\subset\Y$;
  \item[(2)] $\Hom_\DA(\x,\y[-1])=0$ for any $\x\in\X$ and $\y\in\Y$;
  \item[(3)] for any $\c\in\DA$, there is a triangle
  $$\x\rightarrow \c\rightarrow \y[-1] \rightarrow \x[1]$$
  with $\x\in\X$ and $\y\in\Y$.
\end{description}

Silting complexes give rise to $t$-structures in a natural way.
For an integer $m$, consider the pair of subcategories

$$D^{\leq m}(\p) = \{\x\in D^b(A)\mid \Hom_{\DA}(\p,\x[i])=0\text{, for $i>m$}\}$$

\noindent and

$$D^{\geq m}(\p) = \{\x\in \DA \mid \Hom_{\DA}(\p,\x[i])=0\text{, for $i<m$}\}$$

\noindent in the derived category $\DA$.

Observe that $\T(\p) = D^{\leq 0}(\p) \cap \mod A$ and $\F(\p) =  D^{\geq 1}(\p) \cap \mod A$.
We have the following result. Here, (b) is from \cite{hkm} and
(a) is from \cite{ky}.
Note also that a version of (a) was proved in \cite{hkm}, in the setting
of abelian categories with arbitrary coproducts.

\begin{theorem}\label{thm:t-structure}
Let $\p$ be a 2-term silting complex in $K^b(\proj A)$.
\begin{itemize}
\item[(a)] The pair $(D^{\leq0}(\p),D^{\geq0}(\p))$ is a $t$-structure in $\DA$.
\item[(b)] The pair $(\T(\p), \F(\p))$ is a torsion pair in $\mod A$.
\end{itemize}
\end{theorem}

The following lemma will be useful for later.

\begin{lemma}\label{lem:CPdot}
For any $\x\in\DA$ and $i\in\mathbb{Z}$, there is a short exact
sequence, $$0\rightarrow\Hom_{\DA}(\p,H^{i-1}(\x)[1])\rightarrow\Hom_{\DA}(\p,\x[i])\rightarrow
\Hom_{\DA}(\p,H^i(\x))\rightarrow0.$$
\end{lemma}

\begin{proof}
See \cite[Lemma~2.5]{hkm}, the proof given there works also in our case.
\end{proof}

Let $\C(\p) = D^{\leq 0}(\p) \cap D^{\geq 0}(\p)$ be the heart of the $t$-structure $(D^{\leq0}(\p),D^{\geq0}(\p))$.
The following summarizes the main features of $\C(\p)$.

\begin{theorem}\label{thm-cp}
Let $\p$ be a 2-term silting complex in $K^b(\proj A)$.
\begin{itemize}
\item[(a)] $\C(\p)$ is an abelian category and the short exact
  sequences in $\C(\p)$ are precisely the triangles in $D^b(A)$ all of
  whose
vertices are objects in $\C(\p)$.
\item[(b)] $(\F(\p)[1], \T(\p))$ is a torsion pair in $\C(\p)$.
\item[(c)] For a complex $\x$ in $\DA$, we have that  $\x$ is in $\C(\p)$ if and only if $H^0(\x)$ is in $\T(\p)$,
$H^{-1}(\x)$ is in $\F(\p)$ and $H^i(\x) = 0$ for $i \neq -1, 0$.
\item[(d)] $\Hom_{\DA}(\p, - )\colon \C(\p) \to \mod B$ is an equivalence of (abelian) categories.
\end{itemize}
\end{theorem}

\begin{proof}
Note that (a) is a classical result of
\cite{bbd}.  Proofs of (b), (c) and (d) can be found in \cite{hkm} (although there they proved these in the setting of abelian categories with arbitrary coproducts, but their proofs also work in our case, using  Theorem~\ref{thm:t-structure} (a)). We now explain how (b) and (c) can also be seen to follow from \cite[Proposition
I.2.1 and Corollary I.22]{hrs}, which says that for any torsion pair
$(\T,\F)$ in $\mod A$, we have that the two subcategories
$$\{\x\in D^b(A)\mid  H^i(\x) = 0 \text{ for $i>0$
   and } H^0(\x) \in \T  \} $$
\noindent and
$$\{\x\in D^b(A)\mid  H^i(\x) = 0 \text{ for $i<-1$
   and } H^{-1}(\x) \in \F  \} $$
\noindent form a $t$-structure, and that
$(\F[1], \T)$ is a torsion pair in the heart of this $t$-structure.

Note first that by Lemma \ref{lem:CPdot} we have that
\begin{multline*}
D^{\leq0}(\p) =
\{ \x\in D^b(A)\mid \Hom_{\DA}(\p,H^i(\x)) =0 \text{ for $i>0$ and } \\
\Hom_{\DA}(\p,H^j(\x) [1]) =0 \text{, for $j \geq 0$}\}.
\end{multline*}

Since for any module $M$, we have that $\Hom_{\DA}(\p,M) = 0 =
\Hom_{\DA}(\p,M[1])$ only if $M=0$, it follows that
\[\begin{array}{rcl}D^{\leq0}(\p) &=& \{\x\in D^b(A)\mid  H^i(\x) = 0 \text{ for $i>0$
   and } \Hom_{\DA}(\p,H^0(\x) [1]) =0 \}\\
&=&\{\x\in D^b(A)\mid  H^i(\x) = 0 \text{ for $i>0$
   and } H^0(\x) \in \T(\p)  \}. \end{array}\]

Similarly, we have that
$$D^{\geq0}(\p) = \{\x\in D^b(A)\mid  H^i(\x) = 0 \text { for $i<-1$
   and } H^{-1}(\x) \in \F(\p)  \}. $$
Hence (b) and (c) follows.

We also refer to \cite[Proposition~3.13]{iy} for a different proof of (d).
\end{proof}

As before, for a module $M$ in $\mod A$, we let $\Fac M$ denote the full subcategory whose
objects are generated by $M$,  and dually we let $\Sub M$ denote the full subcategory whose
objects are cogenerated by $M$. We then have the following, which is
also due to \cite{hkm}.

\begin{proposition}\label{prop:alt-des}
Let $\p$ be a 2-term silting complex in $K^b(\proj A)$.
Then, we have $$(\T(\p), \F(\p)) = (\Fac H^0(\p), \Sub H^{-1}(\nu(\p))).$$
\end{proposition}

Note that $H^0(\p)$ is the support $\tau$-tilting module corresponding to $\p$ by \cite[Theorem~3.2]{air}, so Proposition~\ref{prop:alt-des} shows that $(\T(\p),\F(\p))$ is precisely the functorially finite torsion pair associated with $H^0(\p)$ via \cite[Theorem~2.7]{air}. 

Consider now the subcategories $\X(\p)=\Hom_{\DA}(\p,\F(\p)[1])$ and $\Y(\p)=\Hom_{\DA}(\p, \T(\p))$
of $\mod B$. We have the following direct consequences of Theorem \ref{thm-cp}.

\begin{corollary}\label{cor:equiv}
Let $\p$ be a 2-term silting complex in $K^b(\proj A)$, then $(\X(\p),\Y(\p))$ is a torsion pair in $\mod B$ and there are equivalences
    $$\Hom_{\DA}(\p, -)\colon \T(\p)\rightarrow \Y(\p),$$
    and
    $$\Hom_{\DA}(\p, -[1]) \colon \F(\p)\rightarrow\X(\p).$$
The equivalences send short exact sequences with terms in $(\T(\p))$ (resp. $\F(\p)$) to short exact sequences in $\mod B$.
\end{corollary}

\begin{proof}
This follows from Theorem \ref{thm-cp} (a) and (d), using that $\T(\p) \cup \F(\p)[1]
\subset \C(\p)$.
\end{proof}

In Section \ref{sec:tst} we will provide natural quasi-inverses
of these functors.

\begin{corollary}\label{cor:hom}
Let $M\in\T(\p)$ and $N\in\F(\p)$, for a 2-term silting complex $\p$.
Then we have the following functorial isomorphisms
\[\Hom_B(\Hom_\DA(\p,M),\Hom_{\DA}(\p,N[1])) \cong \Hom_{\DA}(M,N[1])\cong\Ext^1_A(M,N)\]
and
\[\Ext^1_B(\Hom_{\DA}(\p,M),\Hom_{\DA}(\p,N[1]))\cong\Hom_{\DA}(M,N[2])\cong\Ext^2_A(M,N).\]
\end{corollary}

\begin{proof}
Note that by Theorem \ref{thm-cp} (c) both $M$ and $N[1]$ are in $\C(\p)$.
Then the first isomorphism follows from Theorem \ref{thm-cp} (d), while the second
follows from (a) and (d).
\end{proof}

The following easy observation will be useful later.

\begin{lemma}\label{lem:funciso}
For any $A$-module $X$, we have a functorial isomorphism
\[\Hom_\DA(\p,X)\cong\Hom_A(H^0(\p),X)\]
and a monomorphism
$$\Hom_{\DA}(H^0(\p), X[1]) \to \Hom_{\DA}(\p, X[1]).$$
\end{lemma}

\begin{proof}
Note that for any 2-term complex $\y$ in $D^b(A)$, there is a triangle
$$H^{-1}(\y)[1] \to \y \to H^0(\y) \to H^{-1}(\y)[2].$$
Now applying $\Hom_{\DA}(-,X)$ to the
triangle
\[H^{-1}(\p)[1] \to \p \to H^0(\p) \to H^{-1}(\p)[2]\]
and using that there is no non-zero negative extensions between modules, we get the required isomorphism and monomorphism.
\end{proof}

We next describe some useful properties for the torsion pair corresponding to a 2-term
silting complex. A consequence of this is that both $\T(\p)$ and
$\F(\p)$ are exact categories with enough projectives and injectives. Note that since Proposition~\ref{prop:alt-des} implies that $(\T(\p),\F(\p))$ is precisely the torsion pair associated with the support $\tau$-tilting $A$-module $H^0(\p)$, statement (1) also follows from the proof of \cite[Theorem~2.7]{air}.

\begin{proposition}\label{prop:Ext-proj}
Let $\p$ be a 2-term silting complex and $(\T(\p),\F(\p))$ be the torsion pair induced by $\p$. Then
\begin{description}
\item[(1)] for any $X\in\mod A$, we have that $X\in\add H^0(\p)$ if and only if $X$ is $\Ext$-projective in $\T(\p)$;
\item[(2)] for any $X\in\T(\p)$, there is a short exact sequence
\[0\rightarrow L\rightarrow T_0\rightarrow X\rightarrow0\]
with $T_0\in\add H^0(\p)$ and $L\in\T(\p)$;
\item[(3)] for any $X\in\mod A$,  we have that $X\in \add t\nu A$ if and only if $X$ is $\Ext$-injective in $\T(\p)$;
\item[(4)] for any $X\in\T(\p)$, there is a short exact sequence
\[0\rightarrow X\rightarrow T_0\rightarrow L\rightarrow 0\]
with $T_0\in\add t\nu A$ and $L\in\T(\p)$;
\item[(5)] for any $X\in\mod A$,  we have that $X\in\add H^{-1}(\nu \p)$ if and only if $X$ is $\Ext$-injective in $\F(\p)$;
\item[(6)] for any $X\in\F(\p)$, there is a short exact sequence
\[0\rightarrow X\rightarrow F_0\rightarrow L\rightarrow0\]
with $F_0\in\add H^{-1}(\nu \p)$ and $L\in\F(\p)$;
\item[(7)] for any $X\in\mod A$,  we have that $X\in \add A/tA$ if and only if $X$ is $\Ext$-projective in $\F(\p)$;
\item[(8)] for any $X\in\F(\p)$, there is a short exact sequence
\[0\rightarrow L\rightarrow F_0\rightarrow X\rightarrow 0\]
with $F_0\in\add A/tA$ and $L\in\F(\p)$.
\end{description}
\end{proposition}

\begin{proof}
We only prove $(1)-(4)$. The proofs of $(5)-(8)$ are similar.

By the monomorphism in Lemma \ref{lem:funciso}, we have that $\add H^0(\p)$ is Ext-projective.

Assume $M$ is Ext-projective in $\T(\p)= \Fac H^0(\p)$.
Then there is an exact sequence
\begin{equation}\tag{$\sharp$}
 0\rightarrow L\rightarrow T_0\xrightarrow{\alpha} M\rightarrow0
 \label{short}
 \end{equation}
 where $T_0\xrightarrow{\alpha} M$ is a right
 $\add H^0(\p)$-approximation.
 Since $\Hom_A(H^0(\p),\alpha)$ is an epimorphism, we have that $\Hom_\DA(\p,\alpha)$ is also an epimorphism by Lemma~\ref{lem:funciso}.
     Applying $\Hom_\DA(\p,-)$ to \eqref{short}, we have an exact sequence
$$
\Hom_\DA(\p,T_0) \xrightarrow{\Hom_\DA(\p,\alpha)}\Hom_\DA(\p,M)\rightarrow\Hom_\DA(\p,L[1])\rightarrow0.$$
Then $\Hom_\DA(\p,L[1])=0$ which implies that $L$ is in $\T(\p)$.
Then, by assumption, the sequence \eqref{short} splits, and hence
$M$ is in $\add H^0(\p)$. This proves (1).
Replacing $M$ with an arbitrary object $X$ in $\T(\p)$, we
 also obtain (2). 

For (3) cf. \cite{sma} or \cite[Proposition~VI.1.11]{ass}.

We now prove (4).
For any $X\in\T(\p)$, we have an injective envelope
  $\alpha\colon X\rightarrow I_0$ with $I_0\in\add\nu A$. Considering the canonical exact sequence of $I_0$ in $(\T(\p),\F(\p))$:
    \[\xymatrix@R=1pc{
    &&X\ar[d]^\alpha\ar@{-->}[ld]_{\alpha'}\\
    0\ar[r]&tI_0\ar[r]_\beta&I_0\ar[r]_\gamma&I_0/tI_0\ar[r]&0
    },\]
    we have that $\gamma\alpha=0$ by $X\in\T(\p)$ and $I_0/tI_0\in\F(\p)$. So
    there is a morphism $\alpha'\colon X\rightarrow tI_0$ such that $\alpha=\beta\alpha'$. Note that $\alpha'$ is injective since $\alpha$ is injective. Let $F_0=tI_0\in\add t\nu A$ and $L$ be the cokernel of $\alpha'$. Then $L$ is in $\T(\p)$, since $\T(\p)$ is closed under taking factor modules.
\end{proof}

\section{2-term silting complexes}\label{sec:theory}

The first lemma is the analog, for 2-term silting complexes, of the Bongartz completion of classical tilting modules.
It can be deduced from \cite[Theorem~2.10]{air} and was proven in \cite{a,df,ijy,w}. We state the proof in \cite{a} here for self-containedness.

\begin{lemma}\label{lem:completion}
Let $\p$ be a 2-term presilting complex in $K^b(\proj A)$. Then there exists a triangle
\[A\rightarrow \e \rightarrow \p''\rightarrow A[1]\]
with $\e$ being a 2-term complex in $K^b(\proj A)$ such that $\p\amalg  \e$ is a 2-term silting complex.
\end{lemma}

\begin{proof}
Let $\p''\rightarrow A[1]$ be a right $\add \p$-approximation of $A[1]$. Extend it to a triangle
\begin{equation}\tag{$\ast$}
A\rightarrow \e \rightarrow \p''\rightarrow A[1].
\end{equation}
By applying the functors $\Hom_{\DA}(\p,-)$ and $\Hom_{\DA}(-,\p)$ to the triangle $(\ast)$, we have that $\Hom_{\DA}(\p,\e[i])=0$ for $i>0$ and $\Hom_{\DA}(\e,\p[i])=0$ for $i>0$. Applying $\Hom_{\DA}(-,\e)$ yields $\Hom_{\DA}(\e,\e[i])=0$ for $i>0$. Hence $\p\oplus \e$ is a 2-term presilting complex in $K^b(\proj A)$. The triangle $(\ast)$ shows that $A\in\thick (\p\oplus \e)$ and so $\thick (\p\oplus \e)=K^b(\proj A)$ which implies that $\p\oplus \e$ is a silting complex.
\end{proof}

\begin{remark}
Note that if the right $\add \p$-approximation in the triangle $(\ast)$ is minimal, then $\e$ does not contain any direct summands whose 0th cohomology is zero since $\Hom_{\DA}(A,A[1])=0$. Therefore one can deduce \cite[Theorem~2.10]{air} from this proof.
\end{remark}

We obtain the following characterization of silting complexes.

\begin{corollary}\label{cor:tri}
Let $\p$ be a 2-term presilting complex in $K^b(\proj A)$. Then the following are equivalent:
\begin{description}
  \item[(1)] $\p$ is a silting complex in $K^b(\proj A)$;
  \item[(2)] $|\p|=|A|$;
  \item[(3)] there is a triangle $\Delta_{\p}$
  \[A\s{e}\rightarrow \p' \s{f}\rightarrow \p'' \s{g}\rightarrow A[1]\]
with $\p' , \p'' \in\add \p.$
\end{description}
\end{corollary}

\begin{proof}
The equivalence between (1) and (2) is exactly \cite[Proposition~3.3]{air}, c.f. also \cite{ai,df,ijy}. The equivalence between (1) and (3) is an immediate corollary of Lemma~\ref{lem:completion}, cf. also \cite[Theorem~3.5, Proposition~3.9]{w}.
\end{proof}

\begin{remark}\label{unique-Q}
The map $f$ in the triangle $\Delta_\p$ in Corollary~\ref{cor:tri} defines a 2-term complex in
$K^b(\add \p)$, the bounded homotopy category of the additive category $\add \p$. Since $\Hom_{\DA}(\p,\p[1]) =0$, it follows that $e$ is a left $\add \p$-approximation, and that $g$ is a right $\add \p$-approximation. Moreover, the map $e$ is left minimal if and only if $g$ is right minimal. Hence the resulting 2-term complex is unique up to homotopy equivalence of complexes. Applying $\Hom_{\DA}(\p,-)$ to it, we obtain
a unique 2-term complex $\q$ in $K^b(\proj B)$
\end{remark}

The following lemmas will be useful later.
%
%\begin{lemma}\label{lem:sil}
%Let $\p$ be a 2-term silting complex in $K^b(\proj A)$. Then for a %projective $A$-module $P_0$, $P_0[1]\in\add \p$ if and only if %$\Hom_A(P_0,H^0(\p))=0$.
%\end{lemma}
%
%\begin{proof}
%It follows from \cite[Theorem~2.3]{air}.
%\end{proof}

\begin{lemma}\label{lem:funciso2}
There is a functorial isomorphism
\[\Hom_{\DA}(\p_0,\x)\cong\Hom_B(\Hom_{\DA}(\p,\p_0),\Hom_{\DA}(\p,\x))\]
sending $f$ to $\Hom_\DA(\p,f)$, for any $\p_0\in\add \p$ and $\x\in\DA$.
\end{lemma}

\begin{proof}
This follows from the additivity of the functors and from the fact that the defined map is an isomorphism when $\p_0=\p$.
\end{proof}

%\begin{lemma}\label{lem:card}
%Let $P_1,\cdots,P_n$ be the indecomposable non-isomorphic projective %$A$-modules. Then the modules $P_i/tP_i$'s with $P_i\notin\add \p$ are the %non-isomorphic indecomposable direct summands of $A/tA$. In particular, we %have $\left| A/tA \right|=\left|A\right|-\left|\{P_i\mid P_i\in\add \p\}
%\right|$.
%\end{lemma}
%
%\begin{proof}
%We claim that $P_i/tP_i\cong0$ if and only if $P_i\in\add \p$. Assume that %$P_i/tP_i=0$. This implies $P_i\in\T$, and hence $P_i$ is Ext-projective in %$\T$. So $P_i\in\add H^{(\p)}$ by Proposition~\ref{prop:Ext-proj} (1), and %hence $P_i\in\add \p$. Assume that $P_i\in\add \p$. Then $P_i\in\add %H^0(\p)\subset\T$ and hence $P_i/tP_i=0$. Thus we complete the proof of the %claim. Finally, if $P_i/tP_i\not\cong 0$, then it is indecomposable since %the indecomposable projective module $P_i$ is its projective cover.
%\end{proof}

\begin{lemma}\label{lem:iso}
For each $A$-module $X$, there are isomorphisms
\[\Hom_\DA(\p,X)\cong\Hom_\DA(\p,tX)\]%\cong\Hom_A(H^0(\p),tX)\]
and
\[\Hom_\DA(\p,X[1])\cong\Hom_\DA(\p,X/tX[1])\]%\cong\Hom_A(H^{-1}(\p),X/tX)\]
as $B$-modules, where $0\to tX\to X\to X/tX\to 0$ is the canonical sequence of $X$ with respect to the torsion pair $(\T(\p),\F(\p))$. 
In particular, $\Hom_\DA(\p,X)$ is in $\Y(\p)$ and $\Hom_\DA(\p,X[1])$ is in $\X(\p)$ for any $X$ in $\mod A$.
\end{lemma}

\begin{proof}
Applying $\Hom_\DA(\p,-)$ to the canonical exact sequence of $X$ in \sloppy the torsion pair $(\T(\p),\F(\p))$, we have a long exact sequence
\[\begin{array}{ccccccccc}
0&\rightarrow&\Hom_\DA(\p,tX)&\rightarrow&\Hom_\DA(\p,X)&\rightarrow&\Hom_\DA(\p,X/tX)&&\\
&\rightarrow&\Hom_\DA(\p,tX[1])&\rightarrow&\Hom_\DA(\p,X[1])&\rightarrow&\Hom_\DA(\p,X/tX[1])&\rightarrow&0.
\end{array}\]
Note that $\Hom_\DA(\p,X/tX)=0$ by $X/tX\in\F(\p)$ and $\Hom_\DA(\p,tX[1])=0$ by $tX\in\T(\p)$. Thus we get the desired isomorphisms.
\end{proof}

\begin{lemma}\label{lem:zero}
For any 2-term complex $\x: X^{-1}\xrightarrow{x} X^0$ in $K^b(\proj A)$, if $H^0(\x)\cong 0\cong H^{-1}(\nu\x)$, then $\x\cong 0$.
\end{lemma}

\begin{proof}
On the one hand, $H^0(\x)\cong 0$ implies that $x$ is an epimorphism, so $x$ is a retraction. On the other hand, $H^{-1}(\nu\x)\cong0$ implies that $\nu x$ is a
monomorphism, so $\nu x$ is a section. Since $\nu$ is an equivalence from $\proj A$ to $\inj A$, we have that $x$ is an isomorphism. Hence $\x\cong 0$.
\end{proof}

Recall from Remark \ref{unique-Q} that $\p$ determines a unique (up to isomorphism)
2-term complex $\q$ in $K^b(\proj B)$ given by
\[\Hom_\DA(\p,\p')\xrightarrow{\Hom_\DA(\p,f)}\Hom_\DA(\p,\p''),\]
where $f$ is the map from the triangle $\Delta_{\p}$.

\begin{proposition}\label{prop:newsilt}
Let $\p$ be a 2-term silting complex. Then the complex $\q$ defined above
is a 2-term silting complex in $K^b(\proj B)$. Moreover, $\T(\q)=\X(\p)$ and $\F(\q)=\Y(\p)$.
\end{proposition}

\begin{proof}
Let $P_1,\cdots,P_n$ be a complete collection of indecomposable, pairwise non-isomorphic projective $A$-modules. Since the map $e$ from the triangle $\Delta_{\p}$ is a left $\add\p$-approximation, there are triangles
\[P_i\xrightarrow{e_i}\p'_i\xrightarrow{f_i}\p''_i\xrightarrow{g_i}P_i[1]\]
such that the direct sum of these triangles is a direct summand of $\Delta_{\p}$. Let $\q_i$ be the 2-term complex in $K^b(\proj B)$ given by
\[\Hom_\DA(\p,\p_i')\xrightarrow{\Hom_\DA(\p,f_i)}\Hom_\DA(\p,\p_i''),\]
for each $1\leq i\leq n$. Then $\oplus_{i=1}^n\q_i$ is isomorphic to a direct summand of $\q$. We claim that $\q_1,\cdots,\q_n$ are nonzero and each two of them have no common direct summands. Indeed, by Lemma~\ref{lem:iso}, for each $1\leq i\leq n$,
\[
H^0(\q_i)= \coker\Hom_\DA(\p,f_i)\cong\Hom_\DA(\p,P_i[1])\cong\Hom_\DA(\p,P_i/tP_i [1])
\]
and
\[H^{-1}(\nu\q_i)=\ker\nu\Hom_\DA(\p,f_i)\s{(*)}{\cong}\ker\Hom_\DA(\p,\nu f_i)\cong\Hom_\DA(\p,\nu P_i)\cong\Hom_\DA(\p,t\nu P_i)\]
where
%$\ker\nu\Hom_\DA(\p,f_i)\cong\ker\Hom_\DA(\p,\nu f_i)$
($*$) holds because $\Hom_\DA(\p,\nu \p)\cong D\Hom_\DA(\p,\p)$ is an injective generator of $\mod B$. If $\q_i\cong 0$ for some $i$, both $H^0(\q_i)$ and $H^{-1}(\nu\q_i)$ are isomorphic to zero. Note that $P_i/tP_i\in\F(\p)$ and $t\nu P_i\in\T(\p)$. Then by Corollary~\ref{cor:equiv}, we have $P_i/tP_i\cong 0\cong t\nu P_i$, where the first isomorphism implies that $P_i\in\add\p$, and the second implies that $P_i[1]\in\add\p$. This is a contradiction. Hence $\q_i\ncong 0$.
Note that $P_i$ is a projective cover of $P_i/tP_i$ (if $P_i/tP_i\neq 0$) and $\nu P_i$ is an injective envelope of $t\nu P_i$ (if $t\nu P_i\neq 0$). So by Corollary~\ref{cor:equiv}, for any $i\neq j$, $H^0(\q_i)$ and $H^0(\q_j)$ have no common direct summands, and $H^{-1}(\nu\q_i)$ and $H^{-1}(\nu\q_j)$ have no common direct summands. If $\q_i$ and $\q_j$ have a common direct summand $\x$, then $H^0(\x)\cong0\cong H^{-1}(\nu\x)$. By Lemma~\ref{lem:zero}, $\x\cong0$. We finish the proof of the claim. Therefore, $|\q|\geq|A|$.

To prove that $\q$ is silting, it is by Corollary~\ref{cor:tri}, sufficient to prove that $\q$ is presilting. Let $\alpha$ be a morphism in $\Hom_{K^b(\proj B)}(\q,\q[1])$, then it has the following form
\[\xymatrix@R=1pc{
&&\Hom_\DA(\p,\p')\ar[rr]^{\Hom_\DA(\p,f)}\ar[d]^\alpha&&\Hom_\DA(\p,\p'')\\
\Hom_\DA(\p,\p')\ar[rr]^{\Hom_\DA(\p,f)}&&\Hom_\DA(\p,\p'')
}\]
By Lemma~\ref{lem:funciso2}, there is a morphism $h \colon \p'\rightarrow \p''$ such that $\alpha=\Hom_\DA(\p,h)$. Since \sloppy $\Hom_\DA(A,A[1])=0$, there are morphisms $h_1,h_2$ such that the following commutative diagram is a morphism of triangles :
\[\xymatrix{
A\ar[r]^{e}\ar[d]_{h_1}& \p'\ar[r]^{f}\ar[d]_h& \p''\ar[d]_{h_2}\ar[r]^g&A[1]\ar[d]^{h_1[1]}\\
\p'\ar[r]^{f}&\p''\ar[r]^{g}&A[1]\ar[r]^{-e[1]}&\p'[1]
}\]
By Remark \ref{unique-Q}, the morphism $g$ is a right $\add \p$-approximation of $A[1]$. So there is a morphism $h_3$ such that $h_2=gh_3$. Then $g(h-h_3f)=gh-gh_3f=gh-h_2f=0$. Hence there is a morphism $h_4$ such that $h-h_3f=fh_4$.
\[\xymatrix{
A\ar[r]^{e}\ar[d]_{h_1}& \p'\ar[r]^{f}\ar[d]_h \ar@{-->}[dl]_{h_4} & \p''\ar[d]_{h_2} \ar@{-->}[dl]_{h_3}\ar[r]^g&A[1]\ar[d]^{h_1[1]}\\
\p'\ar[r]^{f}&\p''\ar[r]^{g}&A[1]\ar[r]^{-e[1]}&\p'[1]
}\]
Applying $\Hom_\DA(\p,-)$ to $h-h_3f=fh_4$ yields
\[\alpha=\Hom_\DA(\p, h_3)\Hom_\DA(\p, f)+\Hom_\DA(\p, f)\Hom_\DA(\p, h_4)\]
which implies that $\alpha$, regarded as a map in
$\Hom_\DA(\q,\q[1])$
is null-homotopic. Thus, we have completed the proof that $\q$ is a 2-term silting complex.

Finally we prove that $\T(\q)=\X(\p)$. The proof of $\F(\q)=\Y(\p)$ is similar. Since we have proven that $H^0(\q)\cong\Hom_\DA(\p,A/tA[1])$, by Proposition \ref{prop:alt-des}, it is therefore
sufficient to prove that $\Fac\Hom_\DA(\p,A/tA[1])=\X(\p).$

Let $X$ be in $\X(\p)$. There is then an object $X'$ in $\F(\p)$, such
that $X = \Hom_{\DA}(P, X'[1])$.
By Proposition~\ref{prop:Ext-proj} (8), there is a short exact
sequence
$$0 \to L \to F_0 \to X' \to 0$$
 in $\F(\p)$, with $F_0$ in $\add A/tA$. Then there is an induced triangle $L\to F_0\to X'\to L[1]$ in $D^b(A)$. Apply now $\Hom_\DA(\p,-[1])$ to this triangle, to obtain a short exact sequence in $\mod B$ showing that $X$ is in  $\Fac\Hom_\DA(\p,A/tA[1])$, \sloppy so we have $\X(\p)\subset\Fac\Hom_\DA(\p,A/tA)$. On the other hand, $A/tA\in\F(\p)$ implies $\Hom_\DA(\p,A/tA[1])\in\X(\p)$, hence $\Fac\Hom_\DA(\p,A/tA[1])\subset \X(\p)$, since $\X(\p)$ is closed under factor objects. This concludes the proof.
\end{proof}

\begin{corollary}
The induced torsion pair $(\X(\p),\Y(\p))$ by $\p$ in $\mod B$ is functorially finite.
\end{corollary}

\begin{proof}
This follows from Proposition~\ref{prop:newsilt}, Proposition~\ref{prop:alt-des} and the main result of \cite{sma}.
\end{proof}

\section{A silting theorem}\label{sec:tst}

If $\p$ is a projective presentation of a tilting $A$-module $T$, then $\nu \q[-1]$ is isomorphic to the cotilting $B$-module $D(T)_B = D\Hom_A(T,A)$, and moreover, the endomorphism algebra of this cotilting module is canonically isomorphic to $A$.

It is easy to check that this does not hold in our setting, that is:
in general it does not hold that $\End_{D^b(B)}(\q)$ is isomorphic to $A$,
where $\q$ is the 2-term silting complex in $K^b(\proj A)$, considered
in the previous section. However, we prove that $\End_{D^b(B)}(\q)$ is isomorphic to a factor algebra of $A$ and this factor algebra is isomorphic to $A$ if and only if $\p$ is a tilting complex.
This will then be used to provide mutual equivalences of torsion pairs, as we have in classical tilting theory.

Consider now, as in Remark \ref{unique-Q}, the map
$\p' \s{f}\rightarrow \p''$, coming from the triangle $\Delta_{\p}$ in Corollary \ref{cor:tri}, as an object $\hat{\q}$ in $K^b(\add \p)$,
by letting $\hat{\q}^i = 0$ for all $i \neq -1,0$. Recall that the functor $\Hom_{D^b(A)}(\p,-)$ gives an equivalence between additive categories $\add\p$ and $\proj B$, hence it induces an equivalence of triangulated categories $K^b(\add\p)$ and $K^b(\proj B)$. So it induces an algebra isomorphism $\End_{K^b(\add \p)}(\hat{\q}) \rightarrow \End_{D^b(B)}(\q)$.

We will define an algebra-homomorphism $\End_A(A) \to \End_{K^b(\add \p)}(\hat{\q})$.
For this, represent the object $\p'$ by
$P_{\vartriangle}^{-1} \s{p'}{\rightarrow} P_{\vartriangle}^{0}$,
and represent the object $\p''$ as the mapping cone of $A \to \p'$, that is
$P_{\vartriangle}^{-1} \oplus A \xrightarrow{\left(\begin{smallmatrix}-p' &
e\end{smallmatrix}\right)}{P}_{\vartriangle}^{0}.$

Now, let $a\in\End_A(A)$ and consider the solid diagram, whose rows are triangles,
\[\xymatrix{
\p''[-1]\ar[r]^-{-g[-1]}&A\ar[d]^a\ar[r]^e&\p'\ar@{-->}[d]^b\ar[r]^f&\p''\\
\p''[-1]\ar[r]^-{-g[-1]}&A\ar[r]^e&\p'\ar[r]^f&\p''
}\]
Since $\Hom_\DA(\p''[-1], \p') =0$, there is a map
$b \colon \p' \to \p'$ such that $be = ea$. Choose first such a
map $b=(b_1,b_2)$.
Then, in particular, the following diagram commutes in $\proj A$:
$$\xymatrix@R=1pc{
P_{\vartriangle}^{-1} \ar[rr]^{p'}\ar[dd]^{b_1} && P_{\vartriangle}^{0} \ar[dd]^{b_2} \\
& & \\
P_{\vartriangle}^{-1} \ar[rr]^{p'} && P_{\vartriangle}^{0}
}
$$
Now since the diagram
$$\xymatrix@R=1pc{
A \ar[rr]^{e}\ar[dd]^{a} && \p' \ar[dd]^{(b_1, b_2)} \\
& & \\
A \ar[rr]^{e} && \p'
}
$$
commutes in $K^b(\proj A)$, the chain map 
\[\xymatrix{
0\ar[rr]\ar[d]&&A\ar[d]^{ea-b_2e}\\
P_{\vartriangle}^{-1} \ar[rr]^{p'} && P_{\vartriangle}^{0}
}\]
is null-homotopic. Then there is a map $A \s{t}{\rightarrow} P_{\vartriangle}^{-1}$,
such that $p't = ea - b_2 e$.

Next, consider the endomorphism $c$ of $$P_{\vartriangle}^{-1} \oplus A \xrightarrow{\left(\begin{smallmatrix}-p' &
e\end{smallmatrix}\right)}P_{\vartriangle}^{0}$$
given as follows
$$\xymatrix@R=1pc{
P_{\vartriangle}^{-1} \oplus A  \ar[rr]^{\left(\begin{smallmatrix} -p' & e\end{smallmatrix}\right)} \ar[dd]^{\left(\begin{smallmatrix}  b_1 & t \\ 0 & a \end{smallmatrix}\right)} && P_{\vartriangle}^{0} \ar[dd]^{b_2} \\
& & \\
P_{\vartriangle}^{-1} \oplus A \ar[rr]^{\left(\begin{smallmatrix}-p' & e\end{smallmatrix}\right)} && P_{\vartriangle}^{0}
}
$$
It is straightforward to check, that the map $c$ is a chain map and that we obtain a morphism of triangles
\[\xymatrix@R=1pc{
A \ar[rr]^e
\ar[dd]^a&&\p'\ar[rr]^{f}\ar[dd]^{b}&&\p''\ar[rr]^{g}\ar[dd]^{c}&&A[1]\ar[dd]^{a[1]}\\
\\
A \ar[rr]^e&&\p'\ar[rr]^{f}&&\p''\ar[rr]^{g}&&A[1]
}\]
where $f$ and $g$ now denote the maps
$$\xymatrix@R=1pc{
P_{\vartriangle}^{-1}  \ar[rr]^{p'} \ar[dd]^{\left(\begin{smallmatrix}  -1  \\ 0 \end{smallmatrix}\right)} && P_{\vartriangle}^{0} \ar[dd]^{1}
& & &
P_{\vartriangle}^{-1} \oplus A \ar[rr]^{\left(\begin{smallmatrix}-p' & e\end{smallmatrix}\right)} \ar[dd]^{\left(\begin{smallmatrix}  0 & 1 \end{smallmatrix}\right)} && P_{\vartriangle}^{0} \ar[dd]^{0}
\\
& & & & & \\
P_{\vartriangle}^{-1} \oplus A \ar[rr]^{\left(\begin{smallmatrix}-p' & e\end{smallmatrix}\right)} && P_{\vartriangle}^{0}
& & &
A \ar[rr]^{0} && 0
}
$$

%We need the following observation.
%\begin{lemma}\label{lem:zero}
%Let  $\p' \s{b}{\rightarrow} \p'$ and
% $\p'' \s{c}{\rightarrow} \p''$ be maps in $D^b(A)$ such that
% $(0,b,c)$ is a morphism of triangles, that is we have a commutative %diagram
% \[\xymatrix@R=1pc{
%A \ar[rr]^e
%\ar[dd]^0 %&&\p'\ar[rr]^{f}\ar[dd]^{b}&&\p''\ar[rr]\ar[dd]^{c}&&A[1]\ar[dd]^{0}\\
%\\
%A \ar[rr]^a&&\p'\ar[rr]^{f}&&\p''\ar[rr]&&A[1]
%}\]
%Then $(b,c)$, considered as an endomorphism of $\hat{\q}$, is the zero-map %in $K^b(\add \p)$.
%\end{lemma}

\begin{proposition}\label{prop:silt}
The map $\Phi_{\p} \colon \End_A(A) \to \End_{K^b}(\hat{\q})$ given by
$a \mapsto (b,c)$, where $b=(b_1,b_2)$ and $c=(\left(\begin{smallmatrix}  b_1 & t \\ 0 & a \end{smallmatrix}\right), b_2)$ are chain maps, is a well-defined and surjective algebra morphism with kernel given by
\begin{multline*}
\{v \alpha u \mid u \in \Hom_\DA(A,\p_I) \text{, } \alpha\in\Hom_\DA(\p_I,\p_{II}[-1]) \text{ and } \\ v \in\Hom_\DA(\p_{II}[-1], A) \text{ with } \p_I,\p_{II}\in\add\p\}.
\end{multline*}
Moreover, we have $\ker \Phi_{\p} =0$ if and only if $\Hom_{\DA}(\p,\p[-1]) =0$.
\end{proposition}

\begin{proof}
In order to show that $\Phi_{\p}$ is well-defined, for any map $a\in\End_A(A)$, take two maps $(b^\chi,c^\chi)$, $\chi=1,2$, in $\End_{K^b(\add \p)}(\hat{\q})$ of the form
\[\xymatrix@R=1pc{
 \p'\ar[rr]^{\Big( \left(\begin{smallmatrix} -1 \\  0 \end{smallmatrix}\right), 1 \Big)}\ar[dd]_{(b^\chi_1,b^\chi_2)}&&\p''
 \ar[dd]^{\Big( \left(\begin{smallmatrix} b^\chi_1 & t^\chi \\ 0 & a \end{smallmatrix}\right), b^\chi_2 \Big) } \\
\\
\p'\ar[rr]^{\Big( \left(\begin{smallmatrix} -1 \\  0 \end{smallmatrix}\right), 1 \Big)}&&\p''
}.\]
We need to prove that the two maps $(b^\chi,c^\chi)$ are homotopic. This is equivalent to showing that their difference
\begin{equation}\tag{$\star$}
\xymatrix@R=1pc{
 \p'\ar[rr]^{\Big( \left(\begin{smallmatrix} -1 \\  0 \end{smallmatrix}\right), 1 \Big)}\ar[dd]_{(b^0_1,b^0_2)}&&\p''
 \ar[dd]^{\Big( \left(\begin{smallmatrix} b^0_1 & t^0 \\ 0 & 0 \end{smallmatrix}\right), b^0_2 \Big) } \\
\\
\p'\ar[rr]^{\Big( \left(\begin{smallmatrix} -1 \\  0 \end{smallmatrix}\right), 1 \Big)}&&\p''
}\end{equation}
is null-homotopic in $K^b(\add \p)$, where $b^0_1=b^1_1-b^2_1$, $b^0_2=b^1_2-b^2_2$ and $t^0=t^1-t^2$.

Consider the map $\p'' \s{\mu}{\rightarrow} \p'$
defined as follows:
$$\xymatrix@R=1pc{
P_{\vartriangle}^{-1} \oplus A  \ar[rr]^{\left(\begin{smallmatrix} -p' & e\end{smallmatrix}\right)} \ar[dd]^{\left(\begin{smallmatrix}  -b^0_1 & -t^0  \end{smallmatrix}\right)} && P_{\vartriangle}^{0} \ar[dd]^{b^0_2} \\
& & \\
P_{\vartriangle}^{-1}  \ar[rr]^{p'} && P_{\vartriangle}^{0}
}
$$
Then it is easily verified that $\mu f = (b^0_1,b^0_2)$, and that  $f \mu = \Big( \left(\begin{smallmatrix} b^0_1 & t^0 \\ 0 & 0 \end{smallmatrix}\right), b^0_2 \Big)$
in $\add \p$.
Hence, the map $(\star)$ is null-homotopic and therefore $\Phi_{\p}$ is well-defined. It is easy to check that it is
an algebra homomorphism.

We next show that $\Phi_{\p}$ is surjective.
Consider an arbitrary map $(b,c)$ in $\End_{K^b(\add \p)}(\hat{\q})$ represented by
\[\xymatrix@R=1pc{
 \p'\ar[rr]^{\Big( \left(\begin{smallmatrix} -1 \\  0 \end{smallmatrix}\right), 1 \Big)}\ar[dd]_{(b_1,b_2)}&&\p''
 \ar[dd]^{\Big( \left(\begin{smallmatrix} c_1 & c_2 \\ c_3 & c_4 \end{smallmatrix}\right), c_0 \Big) } \\
\\
\p'\ar[rr]^{\Big( \left(\begin{smallmatrix} -1 \\  0 \end{smallmatrix}\right), 1 \Big)}&&\p''
}\]
It is sufficient to show that
such map is equivalent to a map of the form
\[\xymatrix@R=1pc{
 \p'\ar[rr]^{\Big( \left(\begin{smallmatrix} -1 \\  0 \end{smallmatrix}\right), 1 \Big)}\ar[dd]_{(b_1,b_2)}&&\p''
 \ar[dd]^{\Big( \left(\begin{smallmatrix} b_1 & u \\ 0 & a \end{smallmatrix}\right), b_2 \Big) } \\
\\
\p'\ar[rr]^{\Big( \left(\begin{smallmatrix} -1 \\  0 \end{smallmatrix}\right), 1 \Big)}&&\p''
}\]
for some value of $a$, and for a $u$ satisfying $p'u= e a - b_2 e$. Here, the condition $p'u= e a - b_2 e$, together with $(b_1,b_2)$ being a chain map, ensures that $\Big( \left(\begin{smallmatrix} b_1 & u \\ 0 & a \end{smallmatrix}\right), b_2 \Big)$ is a chain map.

Since $cf = fb$, we have that the following maps
$$\xymatrix@R=1pc{
 P^{-1}\ar[rr]^{p'}\ar[dd]_{\left(\begin{smallmatrix} -c_1 \\  -c_3 \end{smallmatrix}\right)} && P^0
 \ar[dd]^{c_0} & & & P^{-1}\ar[rr]^{p'}\ar[dd]_{\left(\begin{smallmatrix} -b_1 \\  0 \end{smallmatrix}\right)} && P^0
  \ar[dd]^{b_2}
 \\
\\
P^{-1} \oplus A \ar[rr]^{\left(\begin{smallmatrix} -p' & e \end{smallmatrix}\right)}&&P^0 & & &
P^{-1} \oplus A \ar[rr]^{\left(\begin{smallmatrix} -p' & e \end{smallmatrix}\right)}&&P^0
}$$
are homotopic in $K^b(\proj A)$. Hence, there exists
$\left(\begin{smallmatrix} x \\ y \end{smallmatrix}\right)
\colon P^0 \to P^{-1} \oplus A$, such that
$\left(\begin{smallmatrix} x \\ y \end{smallmatrix}\right)p' =
\left(\begin{smallmatrix} c_1-b_1 \\ c_3 \end{smallmatrix}\right)$
and
$b_2-c_0 = \left(\begin{smallmatrix} -p' & e \end{smallmatrix}\right) \left(\begin{smallmatrix} x \\ y \end{smallmatrix}\right)= -p'x +ey$.

It is now straightforward to verify that the map
$c= \Big( \left(\begin{smallmatrix} c_1 & c_2 \\ c_3 & c_4 \end{smallmatrix}\right), c_0 \Big)$ is homotopic to to the map
$\Big( \left(\begin{smallmatrix} b_1 & c_2+xe \\ 0 & c_4+ye \end{smallmatrix}\right), b_2 \Big)$, and that
$u \colon = c_2+xe$ satisfies $p'u= e a - b_2 e$ where $a=c_4+ye$.
This proves the claim, and hence $\Phi_{\p}$ is surjective.
%
% it is sufficient
%to prove that for any $(b,c) = \Big( (b_1,b_2), \left(\begin{smallmatrix} %c_1 & c_2 \\ c_3 & c_4 \end{smallmatrix}\right), c_0 \Big)$, such that the %following diagram
%commutes

Assume now $a$ is in the kernel in $\Phi_{\p}$, so that $(b,c)$ is homotopic to zero. That
is, there exists a chain map $d:\p''\to\p'$ of the following form
$$\xymatrix@R=1pc{
P_{\vartriangle}^{-1} \oplus A  \ar[rr]^{\left(\begin{smallmatrix} -p' & e\end{smallmatrix}\right)} \ar[dd]^{\left(\begin{smallmatrix}  d_1 & d_2  \end{smallmatrix}\right)} && P_{\vartriangle}^{0} \ar[dd]^{w} \\
& & \\
P_{\vartriangle}^{-1}  \ar[rr]^{p'} && P_{\vartriangle}^{0}
}
$$
such that $b=df$ and $c=fd$ in $\add\p$. So $(b_1, b_2)$ is homotopic to $(-d_1,w) = ((d_1,d_2),w)\Big(\left(\begin{smallmatrix}  -1 \\ 0  \end{smallmatrix}\right), 1\Big) $
and 
$\Big(\left(\begin{smallmatrix} b_1 & t \\ 0 & a \end{smallmatrix}\right), b_2\Big)$
is homotopic to
$\Big(\left(\begin{smallmatrix} -d_1 & -d_2 \\ 0 & 0 \end{smallmatrix}\right), w\Big) = \Big(\left(\begin{smallmatrix} -1  \\ 0 & \end{smallmatrix}\right), 1 \Big)((d_1,d_2),w)$.

There is then a map $\delta \colon P_{\vartriangle}^{0} \to P_{\vartriangle}^{-1}$ and
such that
$p' \delta = b_2 -w$ and $\delta p' = b_1 + d_1$, and a map
$\left(\begin{smallmatrix} \epsilon  \\ \theta \end{smallmatrix}\right)\colon
P_{\vartriangle}^{0} \to P_{\vartriangle}^{-1} \oplus A$ such that
$$ \left(\begin{matrix} \epsilon  \\ \theta \end{matrix}\right)
\left(\begin{matrix} -p' & e \end{matrix}\right) =
\left(\begin{matrix} -\epsilon p' & \epsilon e  \\  -\theta p' & \theta e \end{matrix}\right)  =
\left(\begin{matrix} b_1 + d_1 & t + d_2  \\  0 & a \end{matrix}\right)
$$
and such that
$\left(\begin{smallmatrix} -p' & e \end{smallmatrix}\right)
\left(\begin{smallmatrix} \epsilon  \\ \theta \end{smallmatrix}\right)=
-p' \epsilon + e \theta = b_2 -w$.
Combining these equations we obtain
\begin{align*}
 p'(\delta + \epsilon)  &= e \theta  &  (\delta + \epsilon)p'  &= 0.
\end{align*}
Note that in particular we have $\theta p' =0$ and $\theta e = a$.
In this way we obtain that the map $a\colon A \to A$ factors as follows
$$\xymatrix@R=1pc{
& A \ar[dd]^e & &\\
& & &\\
P_{\vartriangle}^{-1}   \ar[r]^{p'} & P_{\vartriangle}^{0}  \ar[dd]^{\left(\begin{smallmatrix}  \delta + \epsilon \\ \theta \end{smallmatrix}\right)} & & \\
& & &\\
& P_{\vartriangle}^{-1} \oplus A  \ar[rr]^{\left(\begin{smallmatrix}  -p' & e  \end{smallmatrix}\right)} \ar[dd]^{\left(\begin{smallmatrix}  0 & 1 \end{smallmatrix}\right)} & &  P_{\vartriangle}^0 \\
& & & \\
& A & &
}
$$
So we have proved that
\begin{multline*}
\ker \Phi_{\p} \subseteq I =\{v \alpha u \mid u \in \Hom_\DA(A,\p_I) \text{, } \alpha\in\Hom_\DA(\p_I,\p_{II}[-1]) \text{ and } \\ v \in\Hom_\DA(\p_{II}[-1], A) \text{ with } \p_I,\p_{II}\in\add\p\}.
\end{multline*}
Next, we prove that $I \subseteq \ker \Phi_{\p}$. Let $a$ be an element in
$I$. Since the map $e \colon A \to \p'$ is a left $\add \p$-approximation, and the
map $g \colon \p'' \to A[1]$ is a right $\add \p$-approximation, we have that
$a = g[-1]ue$ for some map $u \colon \p' \to \p''[-1]$.

Now, assume $u$ is represented by
$P_{\vartriangle}^0 \s{\left(\begin{smallmatrix} u_1 \\ u_2 \end{smallmatrix}\right)}{\rightarrow} P_{\vartriangle}^{-1} \oplus A$, so we have
${\left(\begin{smallmatrix} u_1 p' \\ u_2 p' \end{smallmatrix}\right)}  =
{\left(\begin{smallmatrix} 0 \\ 0 \end{smallmatrix}\right)}$
and
$p'u_1 = eu_2$ and $a = u_2 e$.

Consider the map in $\End_{K^b(\add \p)}(\hat{\q})$ given
by
\[\xymatrix@R=1pc{
\p'\ar[rr]^f\ar[dd]_{(0, e u_2)} && \p''\ar[dd]^{\Big(\left(\begin{smallmatrix}
0 & 0\\ 0 & u_2 e \end{smallmatrix}\right), e u_2\Big)}\\
\\
\p'\ar[rr]^f&& \p''
},\]
Since $a = u_2e$, this map must be homotopic to $\Phi_{\p}(a)$.
The map $(0,eu_2)$ is nullhomotopic in $K^b(\proj A)$, since $u_1 p'= 0$ and
$e u_2= p' u_1$.
Moreover, the map $\Big(\left(\begin{smallmatrix}
0 & 0\\ 0 & u_2 e \end{smallmatrix}\right), eu_2\Big)$ is also
nullhomotopic in $K^b(\proj A)$, since $$\left(\begin{smallmatrix}
0 \\ u_2 \end{smallmatrix}\right) \left(\begin{smallmatrix}
-p' & e  \end{smallmatrix}\right) =
\left(\begin{smallmatrix}
0 & 0 \\ -u_2 p' & u_2 e \end{smallmatrix}\right)
 =
\left(\begin{smallmatrix}
0 & 0 \\ 0 & u_2 e \end{smallmatrix}\right)$$
and
$\left(\begin{smallmatrix} -p' & e  \end{smallmatrix}\right)
 \left(\begin{smallmatrix}
 0 \\ u_2 \end{smallmatrix}\right) = e u_2$.
Hence $a$ is in $\ker \Phi_{\p}$.

We are now left with proving that $\ker \Phi_{\p} =0$ if and only if $\Hom_{\DA}(\p,\p[-1]) =0$. By the first part, we have that $\Hom_{\DA}(\p,\p[-1]) =0$
implies that $\ker \Phi_{\p} =0$.
Assume $\Hom_{\DA}(\p,\p[-1])  \neq 0$.
Then $\Hom_\DA(\p,\p[-1])$ contains a non-zero element $\eta$, which is a chain map:
      \[\xymatrix@R=1pc{
      P^{-1}\ar[r]^{p'}&P^0\ar[d]^\eta\\
      &P^{-1}\ar[r]^{p'}&P^0
      }\]
      So there are $P_i$, $P_j$, indecomposable direct summands of $P^0$, $P^{-1}$ respectively, such that the component of $\eta$ from $P_i$ to $P_j$ is not zero. This induces a non-zero morphism $a_\eta$ in $\Hom_A(A,A)$ which factors through $\eta$. Then $a_\eta$ is in $\ker\Phi_\p$. This concludes the proof.
\end{proof}

\begin{remark}
For any map $a\in\End_A(A)$, its image $(b,c)$ under the epimorphism $\Phi_\p$ makes the diagram
\[\xymatrix@R=1pc{
A \ar[rr]^e
\ar[dd]^a&&\p'\ar[rr]^{f}\ar[dd]^{b}&&\p''\ar[rr]^{g}\ar[dd]^{c}&&A[1]\ar[dd]^{a[1]}\\
\\
A \ar[rr]^e&&\p'\ar[rr]^{f}&&\p''\ar[rr]^{g}&&A[1]
}\]
commute in $K^b(\proj A)$. However, the converse is in general not true. For example, let $A=\k Q/I$, where $Q$ is the following quiver
\[\xymatrix{
1\ar@/^.5pc/[r]^\alpha&2\ar@/^.5pc/[l]^\beta
}\]
and $I$ is generated by $\alpha\beta\alpha$ and $\beta\alpha\beta$. Let $\p=\p_1\oplus\p_2$ with
\[\p_1:\ 0\to P_1\]
and
\[\p_2:\ P_2\to P_1.\]
Then $\p$ is a 2-term silting complex in $K^b(\proj A)$ and we have that $\p'\cong\p_1\oplus\p_1$ and $\p''\cong\p_2$. We take $a$ and $b$ to be zero maps and take $c$ to be the following chain map
\[\xymatrix{
P_2\ar[d]_{\beta\alpha}\ar[rr]&&P_1\ar[d]^0\\
P_2\ar[rr]&&P_1
}\]
It is easily verified that the maps $a$, $b$, $c$ make the diagram commutes. But the pair $(b,c)$ regarded as a chain map from $\hat{\q}$ to itself is not null-homotopic in $K^b(\add\p)$. This means that $(b,c)$ is not $\Phi_\p(a)$.
\end{remark}

The following corollary shows that in the tilting case, our result covers the classical result.

\begin{corollary}\label{cor:tilt}
Under the same notation as before, $\p$ is a tilting complex if and only if $\Phi_{\p}$ is an isomorphism. In this case, $\q$ is also tilting.
\end{corollary}

\begin{proof}
Clearly $\p$ is a tilting complex if and only if $\Hom_{\DA}(\p,\p[-1])= 0$. Hence, the equivalence follows directly from the last part of
Proposition~\ref{prop:silt}.
Assume now $\Hom_{\DA}(\p,\p[-1])= 0$. It suffices to prove
that then also $\Hom_{D^b(B)}(\q,\q[-1])=0$. Note that each morphism $\alpha$ in $\Hom_{D^b(B)}(\q,\q[-1])$ has the following form:
      \[\xymatrix@R=1pc{
      \Hom_\DA(\p,\p')\ar[rr]^{\Hom_\DA(\p,f)}&&\Hom_\DA(\p,\p'')\ar[d]^\alpha\\
      &&\Hom_\DA(\p,\p')\ar[rr]^{\Hom_\DA(\p,f)}&&\Hom_\DA(\p,\p'')
      }\]
      with $\alpha\Hom_\DA(\p,f)=0=\Hom_\DA(\p,f)\alpha$. By Lemma~\ref{lem:funciso2}, there exist $h\colon \p''\rightarrow \p'$ with $\alpha=\Hom_\DA(\p,h)$ and $hf=0=fh$. Hence there exists the following morphism $h_1$:
      \[\xymatrix{
      A\ar[r]^e&\p'\ar@{-->}[ld]_{h_2}\ar[r]^f&\p''\ar@{-->}[ld]_{h_1}\ar[d]^h\ar[r]^g&A[1]\\
      \p''[-1]\ar[r]^{\ \ \ \ -g[-1]}&A\ar[r]^e&\p'\ar[r]^f&\p''
      }\]
      such that $h=eh_1$. Due to $eh_1f=hf=0$, there exists $h_2$ such that $-g[-1]h_2=h_1f$. But $h_2\in\Hom_\DA(\p',\p''[-1])=0$, so $h_1$ factors through $g$ and then it is zero since $\Hom_\DA(A[1],A)=0$. Therefore, $h=0$ which implies that $\alpha=0$. Thus, $\q$ is tilting.
\end{proof}

By now we have proved parts (d) and (e) of Theorem \ref{Main1}, we next
finish the proofs of (f) and (g). Adopting earlier notation, we let
$\X(\q) = \Hom_{D^b(B)}(\q,\F(\q)[1])$
and
$\Y(\q)=\Hom_{D^b(B)}(\q,\T(\q))$. Now, by Corollary \ref{cor:equiv}, we have
that $\Hom_{D^b(b)}(\q,-)$ induces equivalences $\T(\q) \to \Y(\q)$ and
$\F(\q)[1] \to \X(\p)$.

\begin{theorem}\label{thm:silt}
Let $\Phi_\ast \colon \mod\End_{D^b(B)}(\q)\hookrightarrow\mod A$
be the inclusion functor induced by $\Phi_{\p}$. Then $\Phi_\ast(\X(\q))=\T(\p)$ and $\Phi_\ast(\Y(\q))=\F(\p)$.
\end{theorem}

\begin{proof}
We prove that $\Phi_\ast(\Y(\q))=\F(\p)$. The proof of $\Phi_\ast(\X(\q))=\T(\p)$ is similar. By Proposition~\ref{prop:newsilt}, we have that $\T(\q)=\X(\p)$, so we obtain that
 $$\Y(\q)=\Hom_{D^b(B)}(\q,\T(\q))=\Hom_{D^b(B)}(\q,\X(\p))=\Hom_{D^b(B)}(\q,\Hom_{\DA}(\p,\F(\p)[1])).$$
Then to complete the proof, we only need to prove that for any $Y\in\F(\p)$, there is an isomorphism of $A$-modules $Y\cong\Hom_{D^b(B)}(\q,\Hom_{\DA}(\p,Y[1])).$
Note first that
$\Hom_{D^b(B)}(\q,\Hom_{\DA}(\p,Y[1]))$ is the kernel of the map
\begin{multline*}
\Hom_B(\Hom_{\DA}(\p, f), \Hom_{\DA}(\p, Y[1])) \colon
\Hom_B(\Hom_{\DA}(\p,\p''),\Hom_{\DA}(\p,Y[1]))
\to \\
\Hom_B(\Hom_{\DA}(\p,\p'),\Hom_{\DA}(\p,Y[1])).
\end{multline*}
By Lemma~\ref{lem:funciso2}, this is isomorphic to
the kernel of $$\Hom_{\DA}(f,Y[1]) \colon \Hom_{\DA}(\p'',Y[1]) \to \Hom_{\DA}(\p',Y[1]).$$
Applying $\Hom_{\DA}(-,Y[1])$ to the triangle $\Delta_{\p}$, and using
that $\Hom_{\DA}(\p',Y) = 0$, since $Y$ is in $\F(\p)$, we obtain
that
$\ker\Hom_{\DA}(f,Y[1]) \cong \Hom_A(A,Y)$. Hence there is an isomorphism \[\varphi:\ \Hom_A(A,Y)\cong\Hom_{D^b(B)}(\q,\Hom_{\DA}(\p,Y[1]))\]
as vector spaces and for any map $v\in\Hom_A(A,Y)$, the corresponding map $\varphi(v)$ is the following chain map
\[
\xymatrix{
\Hom_\DA(\p,\p')\ar[rr]^{\Hom_\DA(\p,f)}\ar[d]&&\Hom_\DA(\p,\p'')\ar[d]^{\Hom_\DA(\p,v[1]g)}\\
0\ar[rr]&&\Hom_\DA(\p,Y[1])
}
\]
Moreover, for any map $a\in\End_A(A)$, by the commutative diagram
\[\xymatrix@R=1pc{
A\ar[rr]^{e}\ar[dd]^a&&\p'\ar[rr]^{f}\ar[dd]^{b}&&\p''\ar[rr]^{g}\ar[dd]^{c}&&A[1]\ar[dd]^{a[1]}\\
\\
A\ar[rr]^{e}&&\p'\ar[rr]^{f}&&\p''\ar[rr]^{g}&&A[1]
}\]
with $\Phi_{\p}(a) =(b,c)$, we have that $\Hom_\DA(\p,(va)[1]g)=\Hom_\DA(\p,v[1]gc)$. So $\varphi(va)=\Phi_\p(a)\varphi(v)$ which implies that the isomorphism $\varphi$ is a  $A$-module map. Thus, since $Y\cong\Hom_A(A,Y)$, we get the desired isomorphism.

\end{proof}

%\begin{corollary}\label{cor:ker}
%Under the same notation as before, there is an epimorphism of algebras
%\[\End_{\D^b(B)}(\q)\twoheadrightarrow A/(\ann\T(\p)\cap\ann\F(\p)).\]
%\end{corollary}
%
%\begin{proof}
%It is sufficient to prove that %$\ker\Phi(\p)\subset\ann\T(\p)\cap\ann\F(\p)$. But this follows from %that each homomorphism in $\ker\Phi(\p)$ factors through $\p$ and %$\p[1]$.
%\end{proof}

%The epimorphism in Corollary~\ref{cor:ker} is not isomorphic in %general, in other words, $\ker\Phi_{\p}\neq\ann\T(\p)\cap\ann\F(\p)$ in %general, see the first example in Section~\ref{sec:example}.

\section{Auslander-Reiten theory}\label{sec:ar}

As an application of Theorem \ref{Main1}, we show how
the AR-theory of $B = \End_{\DA}(\p)$ can be understood
in terms of the AR-theory of $A$. In the case where $A$ is hereditary,
we obtain particularly strong results. These will turn out
to be essential for studying the so-called {\em silted} algebras,
that is: algebras obtained as $\End_{\DA}(\p)$, for a 2-term
silting complex $\p$ over a hereditary algebra $A$.
Such algebras are investigated and characterized in \cite{bz}.

\subsection{Connecting sequences}
In this section we describe almost split sequences in $\mod B$. Similarly as in classical tilting theory, we call an almost split sequence in $\mod B$ whose left term lies in $\Y(\p)$ and whose right term lies in $\X(\p)$ a {\em connecting sequence}. We denote the AR-translation in a module
category by $\tau$.

%First, we recall the description of the projective and injective $B$-modules.

%\begin{proposition}\label{prop:proj-inj}
%The indecomposable projective $B$-modules are of the form
%$\Hom_\DA(\p,\p_i)$, where $\p_i$ is an indecomposable direct summand
%in $\p$. The corresponding indecomposable injective module is  $\Hom_\DA(\p,\nu \p_i)$.
%\end{proposition}

%\begin{proof}
%The first half of the statement is clear and the second half follows from the Auslander-Reiten formula $\Hom_\DA(\p,\nu \p_i)\cong D\Hom_\DA(\p_i,\p)$, where $D=\Hom_k(-,k)$.
%\end{proof}

\begin{lemma}
If $0\rightarrow Y\rightarrow E\rightarrow X\rightarrow 0$ is a connecting sequence, then there exists an indecomposable projective $A$-module $P_i$ such that $Y\cong \Hom_\DA(\p,\nu P_i)$.
\end{lemma}

\begin{proof}
Since $Y\in\Y(\p)$ and $X= \tau^{-1} Y\in\X(\p)$, by \cite[Lemma~0.1]{sma}, $Y$ is Ext-injective in $\Y(\p)$. Then by Proposition~\ref{prop:Ext-proj} (3), there is an indecomposable $A$-module $Y'\in\add t\nu A$ such that $Y\cong\Hom_\DA(\p,Y')$. Note that for each indecomposable projective $A$-module $P_i$, if $t\nu P_i\not \cong0$, then it is indecomposable since $\nu P_i$ is its injective envelope. So $Y\cong\Hom_\DA(\p,t\nu P_i)$ for some indecomposable projective $A$-module $P_i$. Hence $Y\cong\Hom_\DA(\p,\nu P_i)$ by Lemma~\ref{lem:iso}.
\end{proof}

Note that $\Hom_\DA(\p,\nu P_i)= 0$ if and only if $\nu P_i\in\F(\p)$ if and only if $\nu P_i\in\add H^{-1}(\nu\p)$ if and only if $P_i[1]\in\add\p$. The following lemma is a generalization of the connecting lemma in tilting theory.

\begin{lemma}\label{lem:connecting}
Let $P_i$ be an indecomposable projective $A$-module with $P_i[1]\notin \add\p$. Then
\[\tau^{-1}\Hom_\DA(\p,\nu P_i)\cong\Hom_\DA(\p,P_i[1]).\]
In particular, $\Hom_\DA(\p,\nu P_i)$ is an injective $B$-module if and only if $P_i\in\add \p$.
\end{lemma}

\begin{proof}
This follows from the fact that $\Hom_\DA(\p,P_i[1])\cong H^0(\q_i)$ and $\Hom_\DA(\p,\nu P_i)\cong H^{-1}(\nu\q_i)$ for a 2-term complex $\q_i$ in $K^b(\proj B)$, which was proved in the proof of Proposition~\ref{prop:newsilt}.
\end{proof}

Hence, we have shown that the connecting sequences are
of the form
\[0 \to \Hom_\DA(\p,\nu P_i) \to E \to \Hom_\DA(\p,P_i[1])\to  0.\]
It remains to describe the middle term $E$.

\begin{corollary}\label{cor:forsplit}
Let $P_i$ be an indecomposable projective $A$-module with $P_i\notin \add\p$ and $P_i[1]\notin\add\p$ and $E$ be the middle term of the almost split sequence starting at $\Hom_\DA(\p,\nu P_i)$. Then the canonical sequence of $E$ in the torsion pair $(\X(\p),\Y(\p))$ is
\[0\rightarrow \Hom_\DA(\p,\rad P_i[1])\rightarrow E\rightarrow \Hom_\DA(\p,\nu P_i/S_i)\rightarrow0\]
where $\rad P_i$ denotes the radical of $P_i$ and $S_i$ is the simple
module $P_i/\rad P_i$.
\end{corollary}

\begin{proof}
Since $(\T(\p),\F(\p))$ is a torsion pair, $S_i$ is either in $\T(\p)$ or in $\F(\p)$. We refer to the proof of \cite[Corollary~VI.4.10]{ass} where the first part (i.e. the case $S_i\in\T(\p)$) works in our case by a small suitable modification. However, the second part does not work in our case, instead, one need to use the dual proof of the first part. So for the convenience of readers, we give a proof for the case $S_i\in\F(\p)$.
Applying $\Hom_\DA(\p,-)$ to the triangle
\[\rad P_i\rightarrow P_i\rightarrow S_i\rightarrow \rad P_i[1]\]
yields a short exact sequence
\[0\rightarrow\Hom_\DA(\p,\rad P_i[1])\rightarrow\Hom_\DA(\p,P_i[1])\s{\delta}{\rightarrow}\Hom_\DA(\p,S_i[1])\rightarrow0.\]
Consider the short exact sequences
$$\xymatrix@R=1.6pc{
& & 0 \ar[d] & & \\
& & t\nu P_i \ar[d] & & \\
0 \ar[r]& S_i \ar[r]^{\alpha} & \nu P_i \ar[r]^{\beta} \ar[d]^{\gamma}
 & \nu P_i/S_i \ar[r] & 0\\
& & \nu P_i/ t\nu P_i \ar[d] & & \\
& & 0 & &
}
$$
%\[\begin{array}{ccccccccc}
%&&&&0&&&&\\
%&&&&\downarrow&&&&\\
%&&&&t\nu P_i&&&&\\
%&&&&\downarrow&&&&\\
%0&\rightarrow& S_i&\s{\alpha}\rightarrow& \nu P_i&\s{\beta}\rightarrow& %\nu P_i/S_i&\rightarrow& 0\\
%&&&&\ \ \ \downarrow\gamma&&&&\\
%&&&&\nu P_i/t\nu P_i&&&&\\
%&&&&\downarrow&&&&\\
%&&&&0&&&&
%\end{array}\]
Since $P_i[1] \notin \add\p$, we have that $\gamma \colon \nu P_i \to \nu P_i/t\nu P_i$ is not an isomorphism. Then the composition $\gamma\alpha=0$. So $\gamma$ factors through $\beta$. Because $\Hom_\DA(\p,\gamma[1])$ is an isomorphism by Lemma~\ref{lem:iso}, the map $\Hom_\DA(\p,\beta[1])$ is a monomorphism.
Hence we have a short exact sequence
\[0\rightarrow\Hom_\DA(\p,\nu P_i)\xrightarrow{\theta} \Hom_\DA(\p,\nu P_i/S_i)\rightarrow\Hom_\DA(\p,S_i[1])\rightarrow0\]
where $\theta= \Hom_\DA(\p,\beta)$. Since $\Hom_\DA(\p,S_i[1])\in\X(\p)$ and $\Hom_\DA(\p,\nu P_i/S_i)\in\Y(\p)$ by Lemma~\ref{lem:iso} and $\Hom_\DA(\p,S_i[1])\neq0$ by $S_i\in\F(\p)$, the sequence is not split. In particular, $\theta$ is not a section, so there is a commutative diagram:
\[\xymatrix@R=1pc{
0\ar[r]&\Hom_\DA(\p,\nu P_i)\ar[r]\ar@{=}[dd]&E\ar[r]\ar[dd]&\Hom_\DA(\p,P_i[1])\ar[r]\ar[dd]^h&0\\
\\
0\ar[r]&\Hom_\DA(\p,\nu P_i)\ar[r]^{\theta}&\Hom_\DA(\p,\nu P_i/S_i)\ar[r]&\Hom_\DA(\p,S_i[1])\ar[r]&0
}\]
where the upper sequence is the AR-sequence starting at
$\Hom_\DA(\p,\nu P_i)$. Note that $h \neq 0$, since otherwise
the upper sequence would be split exact.

Since $\Hom_\DA(\p,P_i[1])\cong\Hom_\DA(\p,P_i/tP_i[1])$ by Lemma~\ref{lem:iso} and $\Hom_\DA(\p,-[1])$ is an equivalence from $\F(\p)$ to $\X(\p)$ by Corollary~\ref{cor:equiv}, we have that
\[\Hom_B(\Hom_\DA(\p,P_i[1]),\Hom_\DA(\p,S_i[1]))\cong\Hom_A(P_i/tP_i,S_i).\]
Using that $S_i$ is in $\F(\p)$, by assumption, we have  $\Hom_A(P_i/tP_i,S_i)\cong\Hom_A(P_i,S_i)$, which is a one dimensional space. Therefore, since $h \neq 0$, it equals $k \delta$, for an element $k \in\k$.
Hence, $\ker h \cong \Hom_{\DA}(\p, \rad P_i[1])$.
Using the snake lemma, we obtain the following commutative diagram
\[\xymatrix{
&&0\ar[d]&0\ar[d]\\
&&\Hom_\DA(\p,\rad P_i[1])\ar@{=}[r]\ar[d]&\Hom_\DA(\p,\rad P_i[1])\ar[d]\\
0\ar[r]&\Hom_\DA(\p,\nu P_i)\ar[r]\ar@{=}[d]&E\ar[r]\ar[d]&\Hom_\DA(\p,P_i[1])\ar[r]\ar[d]^h&0\\
0\ar[r]&\Hom_\DA(\p,\nu P_i)\ar[r]^g&\Hom_\DA(\p,\nu P_i/S_i)\ar[r]\ar[d]&\Hom_\DA(\p,S_i[1])\ar[r]\ar[d]&0\\
&&0&0
}\]
where the middle column gives the required short exact sequence.
\end{proof}

\subsection{Separating and splitting silting complexes}

Recall that a torsion pair $(\X,\Y)$ in $\mod A$ is called
{\em split}
(or sometimes splitting) if each indecomposable $A$-module lies either in $\X$ or in $\Y$,
see \cite{ass}. In other words, $(\X,\Y)$ is split if and only if $\Ext_A^1(Y,X)=0$ for all $X\in\X$ and $Y\in\Y$.

\begin{definition}
Let $A$ be a finite dimensional algebra, let $\p$ be a 2-term silting complex in $K^b(\proj A)$ and $B=\End_\DA(\p)$. Then
\begin{description}
  \item[(1)] $\p$ is called {\em separating} if the induced torsion pair $(\T(\p),\F(\p))$ in $\mod A$ is split, and
  \item[(2)] $\p$ is called {\em splitting} if the induced torsion pair $(\X(\p),\Y(\p))$ in $\mod B$ is split.
\end{description}
\end{definition}

\begin{lemma}\label{lem:split}
A 2-term silting complex $\p$ is splitting if and only if $\Ext^2_A(\T(\p),\F(\p))=0$.
\end{lemma}

\begin{proof}
This follows from the second isomorphism in Corollary~\ref{cor:hom}.
\end{proof}

Note that in particular Lemma \ref{lem:split} implies that
if $A$ is hereditary, then all 2-term silting complexes are splitting.
In a forthcoming paper, \cite{bz}, we study endomorphism rings
of 2-term silting complexes over hereditary algebras. We now state a result which is of particular importance for
describing the AR-theory of silted algebras.

\begin{proposition}
If a silting complex $\p$ is splitting, then any almost split sequence in $\mod B$ lies entirely in either $\X(\p)$ or $\Y(\p)$, or else it is of the form
\begin{multline*}0\rightarrow \Hom_\DA(\p,\nu P_i)\rightarrow\Hom_\DA(\p,\rad P_i[1])\oplus\Hom_\DA(\p,\nu P_i/S_i)\rightarrow \\ \Hom_\DA(\p,P_i[1])\rightarrow0,\end{multline*}
where $P_i$ is an indecomposable projective $A$-module with $P_i\notin\add\p$ and $P_i[1]\notin\add \p$. Moreover, almost split sequences
in $\T(\p)$ and $\F(\p)$ are by $\Hom_\DA(\p,-)$ and $\Hom_\DA(\p,-[1])$
mapped to almost
split sequences in $\Y(\p)$ and $\X(\p)$, respectively.
\end{proposition}

\begin{proof}
The first statement follows from Lemma~\ref{lem:connecting} and Corollary~\ref{cor:forsplit}.

For the second statement, we only prove the statement for $\T(\p)$, since the proof for $\F(\p)$ is similar. Let \[0\rightarrow X_1\xrightarrow{\alpha} X_2\xrightarrow{\beta} X_3\rightarrow 0\] be an almost split sequence in $\T(\p)$. Then by Corollary~\ref{cor:equiv}, we have a short exact sequence in $\Y(\p)$:
\[0\rightarrow\Hom_\DA(\p,X_1)\xrightarrow{\Hom_\DA(\p,\alpha)}\Hom_\DA(\p,X_2)\xrightarrow{\Hom_\DA(\p,\beta)}\Hom_\DA(\p,X_3)\rightarrow0\]
where $\Hom_\DA(\p,X_1)$ and $\Hom_\DA(\p,X_3)$ are indecomposable. Let $Y$ be an indecomposable $B$-module, then $Y\in\X(\p)$ or $Y\in\Y(\p)$. To complete the proof, by e.g. \cite[Theorem IV.1.13]{ass}, it is sufficient to prove the following claim: each homomorphism from $Y$ to $\Hom_\DA(\p,X_3)$ which is not a split epimorphism factors through $\Hom_\DA(\p,\beta)$. If $Y\in\Y(\p)$, then this claim follows from that $\Hom_\DA(\p,-)$ is an equivalence from $\T(\p)$ to $\Y(\p)$. Now we assume that $Y\in\X(\p)$. Then $\Hom_{B}(Y,\Hom_\DA(\p,X_3)) = 0$,
so there is nothing left to prove.
\end{proof}

\begin{proposition}\label{lem:sepa}
Each separating 2-term silting complex $\p$ is a tilting complex.
\end{proposition}

\begin{proof}
By Corollary~\ref{cor:tilt}, it is sufficient to prove that $\Phi_\p$ is an isomorphism. This is equivalent to prove that the induced functor $\Phi_\ast \colon \mod\End_{D^b(B)}(\q)\hookrightarrow\mod A$ is an equivalence. Since $\Phi_\ast$ is always fully faithful, we only need to prove that $\Phi_{\ast}$ is surjective. Since $\p$ is separating, each indecomposable $A$-module $M$ is either in $\T(\p)$ or in $\F(\p)$. Then by Theorem~\ref{thm:silt}, there is an $N\in\mod\End_{D^b(B)}(\q)$ such that $\Phi_\ast(N) = M$. Thus, we complete the proof.
\end{proof}

%\begin{proof}
%It is sufficient to prove that $H^{-1}(\p)\in\F(\p).$ Let $\Omega$ be the syzygy functor, i.e., for any $Y$, let $P_1\rightarrow P_0\s{p_Y}\rightarrow Y\rightarrow 0$ be its minimal projective presentation, then $\Omega(Y)=\ker p_Y$. Then for any indecomposable $X\in\add H^{-1}(\p)$, either $X[1]\in\add \p$ or $X\in\add\Omega^2(Y)$ for some $Y\in\add H^0(\p)$. For the first case, $X$ is an indecomposable projective $A-$module with $\Hom_A(X,H^0(\p))=0$ by Lemma~\ref{lem:sil}. So $X/tX$ is an indecomposable object in $\CP$ since its projective cover $X$ in $\CP$ is indecomposable. Then $X\cong X/tX\in\F(\p)$ since $(\T(\p),\F(\p))$ is split. For the second case, consider a minimal projective presentation of $Y$:
%\[P_1\rightarrow P_0\rightarrow Y\rightarrow0.\]
%Since $Y$ is Ext-projective in $\T(\p)$, there are no direct summands of $\Omega(Y)$ which are in $\T(\p)$. Then $\Omega(Y)\in\F(\p)$. So the projective cover $P_1$ of $\Omega(Y)$ is in $\F(\p)$ and hence $\Omega^2(Y)$ is also in $\F(\p)$ due to $\F(\p)$ being closed under submodules.
%\end{proof}

%\begin{corollary}\label{cor:separating}
%If $\p$ is separating, then the homomorphism $\Phi_{\p}$ in Theorem~\ref{thm:silt} is an isomorphism.
%\end{corollary}

%\begin{proof}
%This follows from Lemma~\ref{lem:sepa} and Corollary~\ref{cor:tilt}.
%\end{proof}

%
\end{document}